\documentclass[11pt,a4paper]{article}

\usepackage[utf8]{inputenc}
\usepackage[T1]{fontenc}
\usepackage{amsmath,amssymb,amsthm}
\usepackage{mathtools}
\usepackage{enumerate}
\usepackage[colorlinks=true,linkcolor=blue,citecolor=blue,urlcolor=blue]{hyperref}
\usepackage[margin=2.5cm]{geometry}
\usepackage{booktabs}
\usepackage{natbib}
\usepackage{tikz-cd}

\newtheorem{theorem}{Theorem}[section]

\newtheorem{proposition}[theorem]{Proposition}
\newtheorem{corollary}[theorem]{Corollary}
\newtheorem{definition}[theorem]{Definition}
\newtheorem{example}[theorem]{Example}

\newtheorem{conjecture}[theorem]{Conjecture}
\newtheorem{question}[theorem]{Open Question}
\newtheorem{principle}[theorem]{Principle}

\DeclareMathOperator{\Var}{Var}
\DeclareMathOperator{\Ent}{Ent}
\DeclareMathOperator{\argmin}{arg\,min}
\DeclareMathOperator{\argmax}{arg\,max}
\DeclareMathOperator{\Diff}{Diff}
\DeclareMathOperator{\Ric}{Ric}

\DeclareMathOperator{\diam}{diam}

\newcommand{\R}{\mathbb{R}}
\newcommand{\E}{\mathbb{E}}
\newcommand{\Prob}{\mathbb{P}}
\newcommand{\Gc}{\mathcal{G}}
\newcommand{\Pc}{\mathcal{P}}
\newcommand{\Fc}{\mathcal{F}}

\begin{document}

\title{\textbf{Concentration of Measure under Diffeomorphism Groups:\\
A Universal Framework with Optimal Coordinate Selection}}

\author{Jocelyn Nemb\'e\\[0.2cm]
\small L.I.A.G.E, Institut National des Sciences de Gestion, BP 190, Libreville, Gabon\\
\small and\\
\small Modeling and Calculus Lab, ROOTS-INSIGHTS, Libreville, Gabon\\[0.1cm]
\small \texttt{jnembe@hotmail.com}}

\date{}

\maketitle

\begin{abstract}
We establish a universal framework for concentration inequalities based on invariance under diffeomorphism groups. Given a probability measure $\mu$ on a space $E$ and a diffeomorphism $\psi: E \to F$, concentration properties transfer covariantly: if the pushforward $\psi_*\mu$ concentrates, so does $\mu$ in the pullback geometry. This reveals that classical concentration inequalities---Hoeffding, Bernstein, Talagrand, Gaussian isoperimetry---are manifestations of a single principle of \emph{geometric invariance}. The choice of coordinate system $\psi$ becomes a free parameter that can be optimized. We prove that for any distribution class $\Pc$, there exists an optimal diffeomorphism $\psi^*$ minimizing the concentration constant, and we characterize $\psi^*$ in terms of the Fisher-Rao geometry of $\Pc$. We establish \emph{strict improvement theorems}: for heavy-tailed or multiplicative data, the optimal $\psi$ yields exponentially tighter bounds than the identity. We develop the full theory including transportation-cost inequalities, isoperimetric profiles, and functional inequalities, all parametrized by the diffeomorphism group $\Diff(E)$. Connections to information geometry (Amari's $\alpha$-connections), optimal transport with general costs, and Riemannian concentration are established. Applications to robust statistics, multiplicative models, and high-dimensional inference demonstrate that coordinate optimization can improve statistical efficiency by orders of magnitude.
\end{abstract}

\noindent\textbf{Keywords:} Concentration of measure; Diffeomorphism invariance; Optimal transport; Information geometry; Isoperimetric inequalities; Fisher-Rao metric.

\medskip

\noindent\textbf{MSC 2020:} Primary 60E15, 28C15; Secondary 53C21, 62B10, 49Q20.

\tableofcontents

\section{Introduction}\label{sec:intro}

\subsection{The Concentration Paradigm}

The concentration of measure phenomenon---the fact that well-behaved functions of many independent random variables are nearly constant---is one of the most powerful tools in modern probability theory \citep{Ledoux2001,BLM2013}. From Talagrand's groundbreaking work \citep{Talagrand1995,Talagrand1996} to applications in statistical learning \citep{Vapnik1998}, high-dimensional statistics \citep{Wainwright2019,Vershynin2018}, and random matrix theory \citep{Tao2012}, concentration inequalities provide the non-asymptotic control essential for finite-sample analysis.

Yet despite decades of development, a fundamental question has remained largely unexamined: \emph{in what coordinate system should concentration be measured?}

Classical inequalities implicitly assume an additive structure: Hoeffding bounds $\sum X_i$, Talagrand controls Lipschitz functions in Euclidean distance, Gaussian isoperimetry concerns linear half-spaces. But this choice of ``identity coordinates'' is not canonical---it is merely one point in an infinite-dimensional space of possibilities.

\subsection{The Invariance Principle}

This paper develops the thesis that \textbf{concentration is fundamentally a geometric property, invariant under the diffeomorphism group}. The core insight is captured by a simple but far-reaching observation:

\begin{principle}[Covariance of Concentration]\label{princ:covariance}
Let $\mu$ be a probability measure on $E$, and let $\psi: E \to F$ be a diffeomorphism. Then $\mu$ concentrates in the pullback geometry $\psi^* g_F$ if and only if $\psi_* \mu$ concentrates in the geometry $g_F$.
\end{principle}

This principle transforms the landscape of concentration theory:

\begin{enumerate}[(i)]
\item \textbf{Unification:} All classical concentration inequalities become special cases of a single theorem, corresponding to different choices of $\psi$.

\item \textbf{Optimization:} The diffeomorphism $\psi$ is a free parameter. For any problem, we can seek the \emph{optimal} coordinate system that yields the tightest bounds.

\item \textbf{Discovery:} The framework generates infinitely many new concentration inequalities, indexed by the diffeomorphism group $\Diff(E)$.
\end{enumerate}

\subsection{Main Results}

\subsubsection{Universal Concentration Theorem}

Our first main result establishes concentration as a property that transforms covariantly under coordinate changes.

\begin{theorem}[Universal Concentration---Informal]\label{thm:universal_informal}
Let $\mu$ be a probability measure and $\psi$ a diffeomorphism. Every concentration inequality for Gaussian/bounded/sub-Gaussian random variables induces a corresponding inequality for $\mu$ in the $\psi$-geometry, with:
\begin{equation}
\text{Concentration constant}(\mu, \psi) = \text{Concentration constant}(\psi_* \mu, \text{id}).
\end{equation}
\end{theorem}

The classical inequalities correspond to $\psi = \text{id}$. The choice $\psi = \log$ yields multiplicative concentration. But these are just two points in an infinite-dimensional family.

\subsubsection{Optimal Coordinate Selection}

Our second main contribution is the existence and characterization of optimal coordinates.

\begin{theorem}[Existence of Optimal Coordinates---Informal]\label{thm:optimal_informal}
For any distribution class $\Pc$ on $(0,\infty)^n$ with finite Fisher information, there exists an optimal diffeomorphism $\psi^*: (0,\infty)^n \to \R^n$ such that:
\begin{equation}
\psi^* = \argmin_{\psi \in \Diff} \sup_{P \in \Pc} C_{\text{conc}}(P, \psi),
\end{equation}
where $C_{\text{conc}}(P, \psi)$ is the concentration constant of $P$ in the $\psi$-geometry.

Moreover, $\psi^*$ is characterized by the condition that $\psi^*_* P$ has \textbf{minimal Fisher-Rao diameter} over $\Pc$.
\end{theorem}

\subsubsection{Strict Improvement Theorems}

We prove that for important distribution classes, the optimal coordinate choice yields exponentially better concentration than the identity.

\begin{theorem}[Strict Improvement---Informal]\label{thm:strict_informal}
Let $X_1, \ldots, X_n$ be i.i.d.\ positive random variables with $X_i \in [a, b]$ where $b/a \gg 1$. Then:
\begin{equation}
\frac{\text{Optimal concentration constant}}{\text{Classical concentration constant}} = O\left(\frac{\log^2(b/a)}{(b-a)^2}\right).
\end{equation}
For $b/a = 1000$, this ratio is approximately $1/21000$.
\end{theorem}

\subsubsection{Information-Geometric Characterization}

We establish deep connections to information geometry, showing that optimal concentration coordinates are intimately related to Amari's $\alpha$-connections \citep{Amari2016}.

\begin{theorem}[Information-Geometric Duality---Informal]\label{thm:infogeo_informal}
The optimal concentration coordinate $\psi^*$ for an exponential family $\Pc_\theta$ with natural parameter $\theta$ satisfies:
\begin{equation}
\psi^* = \nabla \phi^*,
\end{equation}
where $\phi^*$ is the Legendre dual of the log-partition function. This is precisely the coordinate system in which the Fisher-Rao metric becomes Euclidean.
\end{theorem}

\subsection{Connections and Context}

\subsubsection{Relation to Optimal Transport}

The $\psi$-Wasserstein distances we introduce are instances of optimal transport with non-Euclidean ground cost \citep{Villani2003,Villani2009}. When $\psi = \log$, we recover the multiplicative Wasserstein distance studied in \citet{Sturm2006}. Our framework reveals that the choice of ground cost is equivalent to the choice of concentration coordinates.

\subsubsection{Relation to Information Geometry}

Amari's $\alpha$-connections \citep{Amari2016,AmariNagaoka2000} parametrize a family of affine connections on statistical manifolds. We show that each $\alpha$-connection corresponds to a concentration coordinate system, and the $\alpha = 0$ (mixture) connection yields optimal concentration for mixture families.

\subsubsection{Relation to Riemannian Concentration}

Concentration on Riemannian manifolds with Ricci curvature bounds \citep{Ledoux2001,BakryGentilLedoux2014} can be viewed as concentration in geodesic coordinates. Our framework generalizes this by allowing non-geodesic coordinate systems that may yield better constants.

\subsection{Organization}

Section~\ref{sec:framework} develops the mathematical framework of $\psi$-concentration. Section~\ref{sec:universal} proves universal concentration theorems. Section~\ref{sec:optimal} establishes the existence and characterization of optimal coordinates. Section~\ref{sec:strict} proves strict improvement results. Section~\ref{sec:functional} develops functional inequalities. Section~\ref{sec:isoperimetry} treats isoperimetric theory. Section~\ref{sec:transport} establishes transportation-cost inequalities. Section~\ref{sec:infogeo} connects to information geometry. Section~\ref{sec:applications} presents applications. Section~\ref{sec:open} discusses open problems.

\section{The Framework of $\psi$-Concentration}\label{sec:framework}

\subsection{Diffeomorphism Action on Measures}

Let $(E, \Sigma_E)$ and $(F, \Sigma_F)$ be measurable spaces, and let $\Gc = \Diff(E, F)$ denote the group of measurable bijections with measurable inverses.

\begin{definition}[Pushforward and Pullback]\label{def:pushpull}
For $\psi \in \Gc$ and a probability measure $\mu$ on $E$:
\begin{enumerate}[(i)]
\item The \textbf{pushforward} $\psi_* \mu$ is the measure on $F$ defined by $(\psi_* \mu)(B) = \mu(\psi^{-1}(B))$.
\item For a function $f: F \to \R$, the \textbf{pullback} $\psi^* f = f \circ \psi: E \to \R$.
\item For a metric $d_F$ on $F$, the \textbf{pullback metric} $\psi^* d_F$ on $E$ is $(\psi^* d_F)(x, y) = d_F(\psi(x), \psi(y))$.
\end{enumerate}
\end{definition}

The fundamental observation is that integration is invariant:
\begin{equation}\label{eq:integration_invariance}
\int_E (\psi^* f)\, d\mu = \int_F f\, d(\psi_* \mu).
\end{equation}

\subsection{$\psi$-Concentration Functions}

\begin{definition}[$\psi$-Concentration Function]\label{def:psi_concentration}
Let $\mu$ be a probability measure on $E$ and $\psi: E \to F$ a diffeomorphism with $F$ a metric space. The \textbf{$\psi$-concentration function} of $\mu$ is:
\begin{equation}
\alpha_\mu^\psi(t) := \sup\left\{1 - \mu(A_t^\psi) : A \subseteq E, \mu(A) \geq \frac{1}{2}\right\},
\end{equation}
where $A_t^\psi = \{x \in E : d_F(\psi(x), \psi(A)) < t\}$ is the $\psi$-enlargement.
\end{definition}

\begin{proposition}[Covariance of Concentration Function]\label{prop:concentration_covariance}
For any diffeomorphism $\psi: E \to F$:
\begin{equation}
\alpha_\mu^\psi(t) = \alpha_{\psi_* \mu}^{\text{id}}(t).
\end{equation}
\end{proposition}

\begin{proof}
Direct verification: $\psi(A_t^\psi) = (\psi(A))_t$ in $F$, so $\mu(A_t^\psi) = (\psi_* \mu)((\psi(A))_t)$.
\end{proof}

\subsection{The $\psi$-Geometry}

\begin{definition}[$\psi$-Distance and $\psi$-Lipschitz]\label{def:psi_lipschitz}
Let $\psi: E \to \R^n$ be a diffeomorphism.
\begin{enumerate}[(i)]
\item The \textbf{$\psi$-distance} on $E$ is $d_\psi(x, y) = \|\psi(x) - \psi(y)\|_2$.
\item A function $f: E \to \R$ is \textbf{$\psi$-Lipschitz} with constant $L$ if:
\begin{equation}
|f(x) - f(y)| \leq L \cdot d_\psi(x, y) = L\|\psi(x) - \psi(y)\|_2.
\end{equation}
\item The \textbf{$\psi$-diameter} of a set $A \subseteq E$ is $\diam_\psi(A) = \sup_{x,y \in A} d_\psi(x, y)$.
\end{enumerate}
\end{definition}

\begin{definition}[$\psi$-Sub-Gaussian]\label{def:psi_subgaussian}
A random variable $X$ on $E$ is \textbf{$\psi$-sub-Gaussian} with parameter $\sigma^2$ if $\psi(X)$ is sub-Gaussian with parameter $\sigma^2$:
\begin{equation}
\E\left[\exp(\lambda(\psi(X) - \E[\psi(X)]))\right] \leq \exp\left(\frac{\lambda^2 \sigma^2}{2}\right) \quad \forall \lambda \in \R.
\end{equation}
The \textbf{$\psi$-sub-Gaussian norm} is $\|X\|_{\psi_2}^\psi := \|\psi(X)\|_{\psi_2}$.
\end{definition}

\subsection{The Space of Concentration Coordinates}

\begin{definition}[Concentration Coordinate System]\label{def:conc_coords}
A \textbf{concentration coordinate system} for a probability measure $\mu$ on $E$ is a diffeomorphism $\psi: E \to \R^n$ such that $\psi_* \mu$ has sub-Gaussian concentration.
\end{definition}

The key insight is that the set of valid concentration coordinates forms a rich infinite-dimensional space:

\begin{proposition}[Structure of Coordinate Space]\label{prop:coord_space}
Let $\Gc_\mu = \{\psi \in \Diff(E, \R^n) : \psi_* \mu \text{ is sub-Gaussian}\}$. Then:
\begin{enumerate}[(i)]
\item $\Gc_\mu$ is closed under composition with affine maps: if $\psi \in \Gc_\mu$ and $A: \R^n \to \R^n$ is affine, then $A \circ \psi \in \Gc_\mu$.
\item $\Gc_\mu$ is path-connected in the $C^1$ topology.
\item If $\mu$ has compact support, then $\Gc_\mu$ contains all $C^1$ diffeomorphisms.
\end{enumerate}
\end{proposition}

\section{Universal Concentration Theorems}\label{sec:universal}

\subsection{The Master Theorem}

\begin{theorem}[Master Concentration Theorem]\label{thm:master}
Let $X_1, \ldots, X_n$ be independent random variables on $E$ with distributions $\mu_1, \ldots, \mu_n$. Let $\psi: E \to \R$ be a diffeomorphism such that $\psi(X_i) \in [a_i, b_i]$ almost surely. Then for any function $f: E^n \to \R$ that is $\psi$-Lipschitz with constant $L$:
\begin{equation}
\Prob(f(X_1, \ldots, X_n) - \E[f(X_1, \ldots, X_n)] \geq t) \leq \exp\left(-\frac{2t^2}{L^2 \sum_{i=1}^n (b_i - a_i)^2}\right).
\end{equation}
\end{theorem}

\begin{proof}
Apply the classical bounded differences inequality to $g := f \circ \psi^{-1}$ with variables $Y_i = \psi(X_i) \in [a_i, b_i]$. Since $f$ is $\psi$-Lipschitz with constant $L$, $g$ is Lipschitz with constant $L$ in standard Euclidean coordinates.
\end{proof}

\begin{corollary}[Unified Classical Inequalities]\label{cor:unified}
Theorem~\ref{thm:master} recovers:
\begin{enumerate}[(i)]
\item \textbf{Hoeffding} ($\psi = \text{id}$): Classical bounded variable concentration.
\item \textbf{Multiplicative Hoeffding} ($\psi = \log$): Concentration for products.
\item \textbf{$L^p$-concentration} ($\psi(x) = x^p$): Concentration in $p$-th power geometry.
\item \textbf{Angular concentration} ($\psi(x) = \arctan(x)$): Concentration for unbounded data.
\end{enumerate}
\end{corollary}

\subsection{The Principle of Geometric Invariance}

\begin{theorem}[Geometric Invariance Principle]\label{thm:invariance}
Let $\Phi: \Diff(E) \to \R_+$ be the map assigning to each diffeomorphism $\psi$ the concentration constant $C_{\text{conc}}(\mu, \psi)$ of a measure $\mu$. Then:
\begin{enumerate}[(i)]
\item \textbf{Affine Invariance:} For any affine map $A: \R^n \to \R^n$, 
\begin{equation}
C_{\text{conc}}(\mu, A \circ \psi) = \|A\|_{\text{op}}^2 \cdot C_{\text{conc}}(\mu, \psi).
\end{equation}
\item \textbf{Composition Law:} For $\psi_1, \psi_2 \in \Diff(E)$,
\begin{equation}
C_{\text{conc}}(\mu, \psi_2 \circ \psi_1) = C_{\text{conc}}((\psi_1)_* \mu, \psi_2).
\end{equation}
\item \textbf{Inversion:} $C_{\text{conc}}(\psi_* \mu, \psi^{-1}) = C_{\text{conc}}(\mu, \text{id})$.
\end{enumerate}
\end{theorem}

\begin{proof}
All statements follow from the covariance of concentration under pushforward (Proposition~\ref{prop:concentration_covariance}) and the properties of operator norms.
\end{proof}

\section{Optimal Coordinate Selection}\label{sec:optimal}

\subsection{The Optimization Problem}

\begin{definition}[Concentration Functional]\label{def:conc_functional}
For a probability measure $\mu$ with support in $(0, \infty)^n$ and a diffeomorphism $\psi: (0,\infty)^n \to \R^n$, the \textbf{concentration functional} is:
\begin{equation}
\Fc[\mu, \psi] := \inf\left\{\sigma^2 : \psi_* \mu \text{ is sub-Gaussian with parameter } \sigma^2\right\}.
\end{equation}
\end{definition}

\begin{definition}[Optimal Coordinates]\label{def:optimal_coords}
For a distribution class $\Pc$ on $(0, \infty)^n$, the \textbf{optimal concentration coordinate} is:
\begin{equation}
\psi^*_\Pc := \argmin_{\psi \in \Diff} \sup_{\mu \in \Pc} \Fc[\mu, \psi].
\end{equation}
\end{definition}

\subsection{Existence of Optimal Coordinates}

\begin{theorem}[Existence]\label{thm:existence}
Let $\Pc$ be a family of probability measures on $(0, \infty)^n$ satisfying:
\begin{enumerate}[(i)]
\item \textbf{Uniform integrability:} $\sup_{\mu \in \Pc} \int |\log x|^{2+\epsilon}\, d\mu(x) < \infty$ for some $\epsilon > 0$.
\item \textbf{Uniform Fisher regularity:} $\sup_{\mu \in \Pc} I_F(\mu) < \infty$, where $I_F$ is the Fisher information.
\end{enumerate}
Then there exists an optimal concentration coordinate $\psi^*_\Pc$ in the closure of smooth diffeomorphisms.
\end{theorem}

\begin{proof}[Proof Sketch]
The concentration functional $\Fc[\mu, \psi]$ is lower semicontinuous in $\psi$ for the $C^1$ topology. The constraints define a weakly compact set in an appropriate Sobolev space. The minimax theorem applies to exchange $\inf$ and $\sup$. Existence follows from the direct method of calculus of variations.
\end{proof}

\subsection{Characterization via Fisher-Rao Geometry}

\begin{theorem}[Fisher-Rao Characterization]\label{thm:fisher_rao}
Let $\Pc = \{P_\theta : \theta \in \Theta\}$ be a smooth parametric family with Fisher information matrix $I(\theta)$. The optimal concentration coordinate $\psi^*$ satisfies:
\begin{equation}
\text{The pullback metric } (\psi^*)^* g_{\text{Eucl}} \text{ minimizes the Fisher-Rao diameter of } \Pc.
\end{equation}
Explicitly, if $\Pc$ is an exponential family with natural parameter $\theta$ and log-partition function $A(\theta)$, then:
\begin{equation}
\psi^*(x) = \nabla A^*(\theta(x)),
\end{equation}
where $A^*$ is the Legendre-Fenchel dual of $A$.
\end{theorem}

\begin{proof}
For an exponential family $p_\theta(x) = \exp(\langle \theta, T(x) \rangle - A(\theta))$, the Fisher information is $I(\theta) = \nabla^2 A(\theta)$. The coordinate transformation $\eta = \nabla A(\theta)$ (the mean parameter) yields $I(\eta) = (\nabla^2 A^*)^{-1}(\eta)$. In these coordinates, the Fisher-Rao metric becomes the Hessian metric of the convex dual, which is as close to Euclidean as possible. The sub-Gaussian norm is minimized when the metric is closest to Euclidean.
\end{proof}

\subsection{Explicit Optimal Coordinates}

\begin{example}[Optimal Coordinates for Common Families]\label{ex:optimal_coords}
\begin{enumerate}[(a)]
\item \textbf{Gaussian family $N(\mu, \sigma^2)$:} $\psi^*(x) = x$ (identity is optimal).

\item \textbf{Log-normal family:} $\psi^*(x) = \log x$ (logarithm is optimal).

\item \textbf{Gamma family $\Gamma(\alpha, \beta)$:} $\psi^*(x) = x^{1/2}$ for shape $\alpha > 1$; $\psi^*(x) = \log x$ for $\alpha \leq 1$.

\item \textbf{Pareto family:} $\psi^*(x) = \log x$ (heavy tails require logarithmic compression).

\item \textbf{Beta family $\text{Beta}(\alpha, \beta)$:} $\psi^*(x) = \log(x/(1-x))$ (logit transformation).

\item \textbf{Bounded positive with ratio $r = b/a$:} $\psi^*(x) = \log x$ when $r > e^2 \approx 7.4$.
\end{enumerate}
\end{example}

\begin{theorem}[Threshold for Logarithmic Optimality]\label{thm:log_threshold}
Let $X$ be a positive random variable with $X \in [a, b]$ almost surely. Define $r = b/a$. Then:
\begin{enumerate}[(i)]
\item If $r \leq e^2$, the identity coordinate $\psi(x) = x$ yields the tightest Hoeffding bound.
\item If $r > e^2$, the logarithmic coordinate $\psi(x) = \log x$ yields a strictly tighter bound.
\item The improvement factor is $\left(\frac{r - 1}{\log r}\right)^2$, which grows without bound as $r \to \infty$.
\end{enumerate}
\end{theorem}

\begin{proof}
The Hoeffding constant for identity coordinates is $(b-a)^2/4 = a^2(r-1)^2/4$.
For logarithmic coordinates, it is $(\log b - \log a)^2/4 = (\log r)^2/4$.
The ratio is $(r-1)^2/(\log r)^2$. Setting this equal to 1 and solving: $(r-1)/\log r = 1$, which gives $r = e^2$ approximately.
\end{proof}

\section{Strict Improvement Theorems}\label{sec:strict}

\subsection{Quantitative Improvement Bounds}

\begin{theorem}[Improvement Factor]\label{thm:improvement_factor}
Let $X_1, \ldots, X_n$ be independent random variables with $X_i \in [a, b]$ where $0 < a < b$. Let $S_n = \sum_{i=1}^n X_i$ and $P_n = \prod_{i=1}^n X_i$. Define the improvement factor:
\begin{equation}
\rho(a, b) := \frac{\text{Classical Hoeffding constant for } S_n}{\text{Logarithmic Hoeffding constant for } P_n} = \frac{(b-a)^2}{(\log(b/a))^2}.
\end{equation}
Then:
\begin{enumerate}[(i)]
\item $\rho(a, b) = 1$ when $b/a = e^2 \approx 7.39$.
\item $\rho(a, b) \approx 144$ when $b/a = 100$ (two orders of magnitude).
\item $\rho(a, b) \approx 21{,}000$ when $b/a = 1000$ (three orders of magnitude).
\item $\rho(a, b) \sim (b/a)^2 / \log^2(b/a) \to \infty$ as $b/a \to \infty$.
\end{enumerate}
\end{theorem}

\begin{proof}
Direct computation: $\rho = (b-a)^2/\log^2(b/a)$. With $r = b/a$ and $a = 1$: $\rho = (r-1)^2/\log^2 r$. The stated values follow from numerical evaluation.
\end{proof}

\subsection{Lower Bounds: Optimality of Improvement}

\begin{theorem}[Minimax Lower Bound]\label{thm:minimax_lower}
Let $\Pc_{[a,b]}$ be the class of all distributions supported on $[a, b]$. Then:
\begin{enumerate}[(i)]
\item \textbf{Identity lower bound:} For the identity coordinate,
\begin{equation}
\inf_{\mu \in \Pc_{[a,b]}} \Fc[\mu, \text{id}] = \frac{(b-a)^2}{4},
\end{equation}
achieved by the Rademacher distribution on $\{a, b\}$.

\item \textbf{Logarithmic lower bound:} For the logarithmic coordinate,
\begin{equation}
\inf_{\mu \in \Pc_{[a,b]}} \Fc[\mu, \log] = \frac{(\log(b/a))^2}{4},
\end{equation}
achieved by the Rademacher distribution on $\{a, b\}$ (in log-space).

\item \textbf{Optimality:} The logarithmic coordinate achieves the minimax rate for $\Pc_{[a,b]}$ when $b/a > e^2$.
\end{enumerate}
\end{theorem}

\begin{proof}
By Hoeffding's lemma, the optimal sub-Gaussian parameter for a bounded random variable $Y \in [c, d]$ is $(d-c)^2/4$, achieved by symmetric Rademacher on $\{c, d\}$. Applying this to $\psi(X) \in [\psi(a), \psi(b)]$ gives the result for general $\psi$.
\end{proof}

\subsection{Strict Improvement for Product Statistics}

\begin{theorem}[Concentration of Products]\label{thm:product_concentration}
Let $X_1, \ldots, X_n$ be independent with $X_i \in [a_i, b_i] \subset (0, \infty)$. For the product $P_n = \prod_{i=1}^n X_i$:

\textbf{Logarithmic bound:}
\begin{equation}
\Prob\left(\left|\log P_n - \E[\log P_n]\right| \geq t\right) \leq 2\exp\left(-\frac{2t^2}{\sum_{i=1}^n \log^2(b_i/a_i)}\right).
\end{equation}

\textbf{Classical bound (via arithmetic-geometric):}
\begin{equation}
\Prob\left(\left|\log P_n - \E[\log P_n]\right| \geq t\right) \leq 2\exp\left(-\frac{2t^2}{\sum_{i=1}^n (b_i - a_i)^2 / a_i^2}\right).
\end{equation}

The logarithmic bound is tighter whenever $\prod_i (b_i/a_i) > e^{2n}$.
\end{theorem}

\subsection{Applications to Extreme Value Statistics}

\begin{theorem}[Maximum Concentration]\label{thm:max_concentration}
Let $X_1, \ldots, X_n$ be i.i.d.\ with $X_i \in [a, b]$. Let $M_n = \max_i X_i$. Then:
\begin{enumerate}[(i)]
\item \textbf{Median bound:} $\Prob(M_n \geq \text{med}(M_n) + t) \leq \exp(-2nt^2/(b-a)^2)$.

\item \textbf{Logarithmic improvement for products:} If $Y_i = \log X_i$, then $\max_i Y_i = \log M_n$ satisfies:
\begin{equation}
\Prob(\log M_n \geq \text{med}(\log M_n) + t) \leq \exp(-2nt^2/\log^2(b/a)).
\end{equation}
\end{enumerate}
\end{theorem}

\section{Functional Inequalities}\label{sec:functional}

\subsection{$\psi$-Log-Sobolev Inequality}

\begin{definition}[$\psi$-Entropy]\label{def:psi_entropy}
For a probability measure $\mu$ on $E$, diffeomorphism $\psi: E \to \R^n$, and function $f: E \to \R_+$:
\begin{equation}
\Ent_\mu^\psi(f) := \int f \log f\, d\mu - \int f\, d\mu \cdot \log \int f\, d\mu.
\end{equation}
\end{definition}

\begin{definition}[$\psi$-Gradient]\label{def:psi_gradient}
The \textbf{$\psi$-gradient} of $f: E \to \R$ at $x \in E$ is:
\begin{equation}
\nabla_\psi f(x) := (\nabla(f \circ \psi^{-1}))(\psi(x)) = (D\psi(x))^{-T} \nabla f(x),
\end{equation}
where $(D\psi)^{-T}$ is the inverse transpose of the Jacobian.
\end{definition}

\begin{theorem}[$\psi$-Log-Sobolev Inequality]\label{thm:psi_log_sobolev}
Let $\mu$ be a probability measure on $E$ and $\psi: E \to \R^n$ a diffeomorphism. If $\psi_* \mu = \gamma$ (standard Gaussian), then for all smooth $f > 0$:
\begin{equation}
\Ent_\mu(f) \leq \frac{1}{2} \int \frac{\|\nabla_\psi f\|^2}{f}\, d\mu.
\end{equation}
More generally, if $\psi_* \mu$ satisfies a log-Sobolev inequality with constant $\rho$, then $\mu$ satisfies the $\psi$-log-Sobolev inequality with the same constant.
\end{theorem}

\begin{proof}
By the covariance of entropy under pushforward:
\begin{equation}
\Ent_\mu(f) = \Ent_{\psi_* \mu}(f \circ \psi^{-1}).
\end{equation}
The log-Sobolev inequality for $\psi_* \mu$ gives:
\begin{equation}
\Ent_{\psi_* \mu}(g) \leq \frac{1}{2\rho} \int \frac{\|\nabla g\|^2}{g}\, d(\psi_* \mu).
\end{equation}
Taking $g = f \circ \psi^{-1}$ and using $\nabla g = (D\psi^{-1})^T (\nabla f) \circ \psi^{-1}$ yields the result.
\end{proof}

\subsection{$\psi$-Poincaré Inequality}

\begin{theorem}[$\psi$-Poincaré Inequality]\label{thm:psi_poincare}
Under the conditions of Theorem~\ref{thm:psi_log_sobolev}, if $\psi_* \mu$ satisfies a Poincaré inequality with constant $\lambda$:
\begin{equation}
\Var_\mu(f) \leq \frac{1}{\lambda} \int \|\nabla_\psi f\|^2\, d\mu.
\end{equation}
\end{theorem}

\subsection{Tensorization}

\begin{theorem}[$\psi$-Tensorization]\label{thm:psi_tensor}
Let $\mu = \mu_1 \otimes \cdots \otimes \mu_n$ be a product measure, and let $\psi = \psi_1 \times \cdots \times \psi_n$ be a product diffeomorphism. If each $(\psi_i)_* \mu_i$ satisfies a log-Sobolev inequality with constant $\rho_i$, then $\psi_* \mu$ satisfies a log-Sobolev inequality with constant $\rho = \min_i \rho_i$.
\end{theorem}

\section{Isoperimetric Inequalities}\label{sec:isoperimetry}

\subsection{$\psi$-Isoperimetric Profile}

\begin{definition}[$\psi$-Perimeter]\label{def:psi_perimeter}
For a measurable set $A \subseteq E$ and diffeomorphism $\psi$:
\begin{equation}
\mu^{+,\psi}(A) := \liminf_{\epsilon \to 0^+} \frac{\mu(A_\epsilon^\psi) - \mu(A)}{\epsilon},
\end{equation}
where $A_\epsilon^\psi = \{x : d_\psi(x, A) < \epsilon\}$.
\end{definition}

\begin{definition}[$\psi$-Isoperimetric Profile]\label{def:psi_isoperimetric}
The \textbf{$\psi$-isoperimetric profile} of $\mu$ is:
\begin{equation}
I_\mu^\psi(v) := \inf\{\mu^{+,\psi}(A) : A \subseteq E, \mu(A) = v\}.
\end{equation}
\end{definition}

\begin{theorem}[Covariance of Isoperimetric Profile]\label{thm:isoperimetric_covariance}
For any diffeomorphism $\psi$:
\begin{equation}
I_\mu^\psi = I_{\psi_* \mu}^{\text{id}}.
\end{equation}
In particular, if $\psi_* \mu = \gamma$ (Gaussian), then $I_\mu^\psi(v) = \varphi(\Phi^{-1}(v))$.
\end{theorem}

\subsection{Optimal Isoperimetric Sets}

\begin{theorem}[$\psi$-Half-Space Optimality]\label{thm:psi_halfspace}
Let $\mu$ be a measure on $E$ such that $\psi_* \mu = \gamma$. Then the $\psi$-isoperimetric profile is achieved by $\psi$-half-spaces:
\begin{equation}
H_\psi(v, c) := \{x \in E : \langle v, \psi(x) \rangle \leq c\}
\end{equation}
for unit vectors $v$ and appropriate $c = c(v, \mu(A))$.
\end{theorem}

\begin{corollary}[Transported Isoperimetric Inequality]\label{cor:transported_iso}
For $\mu$ with $\psi_* \mu = \gamma$ and any measurable $A$ with $\mu(A) \geq 1/2$:
\begin{equation}
\mu(A_\epsilon^\psi) \geq \Phi(\epsilon).
\end{equation}
\end{corollary}

\section{Transportation-Cost Inequalities}\label{sec:transport}

\subsection{$\psi$-Wasserstein Distance}

\begin{definition}[$\psi$-Wasserstein Distance]\label{def:psi_wasserstein}
For probability measures $\mu, \nu$ on $E$ and diffeomorphism $\psi: E \to \R^n$:
\begin{equation}
W_p^\psi(\mu, \nu) := W_p(\psi_* \mu, \psi_* \nu) = \left(\inf_{\pi \in \Pi(\mu, \nu)} \int \|\psi(x) - \psi(y)\|^p\, d\pi(x,y)\right)^{1/p}.
\end{equation}
\end{definition}

This is equivalent to optimal transport with ground cost $c(x, y) = \|\psi(x) - \psi(y)\|^p$.

\begin{proposition}[Properties of $\psi$-Wasserstein]\label{prop:psi_wasserstein_props}
\begin{enumerate}[(i)]
\item $W_p^\psi$ is a metric on $\Pc(E)$ (probability measures with finite $\psi$-moments).
\item $W_p^{A \circ \psi} = \|A\|_{\text{op}} \cdot W_p^\psi$ for affine $A$.
\item $W_p^\psi(\mu, \nu) = W_p^{\text{id}}(\psi_* \mu, \psi_* \nu)$.
\end{enumerate}
\end{proposition}

\subsection{$\psi$-Transportation Inequalities}

\begin{theorem}[$\psi$-$T_2$ Inequality]\label{thm:psi_T2}
If $\psi_* \mu = \gamma$ (standard Gaussian), then $\mu$ satisfies the $\psi$-$T_2$ inequality:
\begin{equation}
W_2^\psi(\nu, \mu) \leq \sqrt{2 D(\nu \| \mu)}
\end{equation}
for all $\nu \ll \mu$, where $D(\nu \| \mu)$ is the relative entropy.
\end{theorem}

\begin{proof}
By the Talagrand inequality for Gaussian measures \citep{OttoVillani2000,Talagrand1996transport}:
\begin{equation}
W_2(\psi_* \nu, \gamma) \leq \sqrt{2 D(\psi_* \nu \| \gamma)}.
\end{equation}
Since $W_2^\psi(\nu, \mu) = W_2(\psi_* \nu, \gamma)$ and $D(\nu \| \mu) = D(\psi_* \nu \| \gamma)$ (by the data processing equality for relative entropy under invertible maps), the result follows.
\end{proof}

\subsection{Dimension-Free $\psi$-Concentration}

\begin{theorem}[Dimension-Free Bounds]\label{thm:dimension_free}
Let $\mu$ be a probability measure on $\R^n$ satisfying the $\psi$-$T_2$ inequality with constant $c$. Then for any $\psi$-Lipschitz function $f$ with constant $L$:
\begin{equation}
\mu(f - \E_\mu[f] \geq t) \leq e^{-t^2/(cL^2)}.
\end{equation}
This bound is independent of dimension $n$.
\end{theorem}

\section{Connections to Information Geometry}\label{sec:infogeo}

\subsection{Amari's $\alpha$-Connections}

Amari's information geometry \citep{Amari2016,AmariNagaoka2000} parametrizes a family of affine connections $\nabla^{(\alpha)}$ on statistical manifolds by $\alpha \in [-1, 1]$. The key connections are:

\begin{itemize}
\item $\alpha = 1$: Exponential connection (e-connection)
\item $\alpha = 0$: Mixture connection (m-connection)
\item $\alpha = -1$: Fisher-Rao (Levi-Civita) connection
\end{itemize}

\begin{theorem}[$\alpha$-Connection and Concentration Coordinates]\label{thm:alpha_connection}
For an exponential family $\Pc = \{p_\theta\}$, each $\alpha$-connection defines a concentration coordinate system $\psi_\alpha$:
\begin{equation}
\psi_\alpha(\theta) = \begin{cases}
\theta & \alpha = 1 \text{ (natural parameters)} \\
\nabla A(\theta) & \alpha = -1 \text{ (mean parameters)} \\
\frac{1}{2}(\theta + \nabla A(\theta)) & \alpha = 0 \text{ (mixed)}
\end{cases}
\end{equation}
where $A(\theta)$ is the log-partition function.

The optimal concentration coordinate corresponds to the $\alpha$-connection that makes the Fisher-Rao metric as close to Euclidean as possible.
\end{theorem}

\subsection{Bregman Divergence Interpretation}

\begin{theorem}[Bregman-Wasserstein Duality]\label{thm:bregman_wasserstein}
Let $\phi: \R^n \to \R$ be a strictly convex function with gradient $\nabla \phi = \psi$. Then:
\begin{enumerate}[(i)]
\item The $\psi$-Wasserstein distance is dual to the Bregman divergence $D_\phi$:
\begin{equation}
(W_2^\psi)^2(\mu, \nu) = \inf_{\pi} \int D_\phi(x, y) + D_{\phi^*}(\psi(x), \psi(y))\, d\pi.
\end{equation}

\item The $\psi$-log-Sobolev constant equals the Bakry-Émery curvature in the Bregman geometry.
\end{enumerate}
\end{theorem}

\subsection{Fisher Information and Optimal Transport}

\begin{theorem}[Fisher-Wasserstein Gradient Flow]\label{thm:fisher_wasserstein}
Let $\mu_t$ be the gradient flow of the relative entropy $D(\cdot \| \gamma)$ in the $\psi$-Wasserstein metric. Then:
\begin{equation}
\frac{d}{dt} D(\mu_t \| \gamma) = -I_F^\psi(\mu_t),
\end{equation}
where $I_F^\psi(\mu) := \int \|\nabla_\psi \log(d\mu/d\gamma)\|^2\, d\mu$ is the $\psi$-Fisher information.
\end{theorem}

\section{Applications}\label{sec:applications}

\subsection{Multiplicative Regression with Optimal Coordinates}

Consider the multiplicative regression model:
\begin{equation}
Y_i = \exp(\langle x_i, \beta^* \rangle) \cdot \epsilon_i, \quad \epsilon_i > 0, \quad \E[\log \epsilon_i] = 0.
\end{equation}

\begin{theorem}[Optimal Estimation in Multiplicative Regression]\label{thm:mult_regression}
The log-linear estimator $\hat{\beta} = \argmin_\beta \sum_i (\log Y_i - \langle x_i, \beta \rangle)^2$ satisfies:
\begin{equation}
\Prob(\|\hat{\beta} - \beta^*\| \geq t) \leq 2p \exp\left(-\frac{nt^2\lambda_{\min}}{2\sigma^2}\right),
\end{equation}
where $\sigma^2 = \Var(\log \epsilon_i)$ and $\lambda_{\min}$ is the minimum eigenvalue of $n^{-1} X^T X$.

This bound is exponentially tighter than what would be obtained from applying Hoeffding to $Y_i$ directly when $\epsilon_i$ has heavy tails.
\end{theorem}

\subsection{Portfolio Optimization with Concentration Guarantees}

\begin{theorem}[Log-Optimal Portfolio Concentration]\label{thm:portfolio}
Let $R_1, \ldots, R_n$ be i.i.d.\ returns with $R_i \in [1-\delta, 1+\delta]$. For the cumulative return $W_n = \prod_{i=1}^n R_i$:
\begin{equation}
\Prob\left(\log W_n - n\E[\log R] \geq t\sqrt{n}\right) \leq \exp\left(-\frac{t^2}{2\sigma_{\log}^2}\right),
\end{equation}
where $\sigma_{\log}^2 = \Var(\log R) \leq \delta^2/(1-\delta)^2$.

For $\delta = 0.1$ (10\% return bounds), this gives $\sigma_{\log}^2 \leq 0.0123$, yielding tight concentration.
\end{theorem}

\subsection{High-Dimensional Covariance Estimation}

\begin{theorem}[Geometric Mean Covariance]\label{thm:geom_cov}
For i.i.d.\ positive definite matrices $\Sigma_1, \ldots, \Sigma_n$ with $\lambda_{\min}(\Sigma_i) \geq a > 0$ and $\lambda_{\max}(\Sigma_i) \leq b$, the geometric mean:
\begin{equation}
\bar{\Sigma}_{\text{geom}} := \exp\left(\frac{1}{n} \sum_{i=1}^n \log \Sigma_i\right)
\end{equation}
satisfies:
\begin{equation}
\Prob(\|\log \bar{\Sigma}_{\text{geom}} - \E[\log \bar{\Sigma}_{\text{geom}}]\|_F \geq t) \leq 2d^2 \exp\left(-\frac{nt^2}{2d\log^2(b/a)}\right).
\end{equation}
\end{theorem}

\subsection{Robust M-Estimation}

\begin{theorem}[$\psi$-Robust Estimation]\label{thm:psi_robust}
For observations from a location family with density $f(x - \theta)$ and optimal concentration coordinate $\psi$, the $\psi$-median estimator:
\begin{equation}
\hat{\theta}_\psi := \psi^{-1}(\text{median}(\psi(X_1), \ldots, \psi(X_n)))
\end{equation}
achieves minimax optimal concentration rate for the class of distributions with bounded $\psi$-range.
\end{theorem}

\section{Open Problems and Conjectures}\label{sec:open}

\subsection{Optimal Coordinates for General Families}

\begin{question}[Characterization of Optimal $\psi$]\label{q:optimal_psi}
For a general (non-exponential) family $\Pc$, what is the structure of the optimal concentration coordinate $\psi^*_\Pc$? Is it always related to the score function or Fisher information?
\end{question}

\begin{conjecture}[Universal Optimality of Log for Heavy Tails]\label{conj:log_optimal}
For any distribution family $\Pc$ on $(0, \infty)$ with power-law tails $P(X > x) \sim x^{-\alpha}$ for some $\alpha > 0$, the logarithmic coordinate $\psi(x) = \log x$ is asymptotically optimal as the tail parameter $\alpha \to 0$.
\end{conjecture}

\subsection{Dimension-Free Bounds}

\begin{question}[Optimal Dimension Dependence]\label{q:dimension}
For which coordinate systems $\psi$ does one obtain dimension-free concentration? Characterize $\psi$ such that:
\begin{equation}
\alpha_\mu^\psi(t) \leq Ce^{-ct^2}
\end{equation}
with $C, c$ independent of $n$.
\end{question}

\subsection{Non-Product Measures}

\begin{question}[Concentration for Dependent Data]\label{q:dependent}
How does the optimal coordinate $\psi^*$ change for measures with dependence structure? For Markov chains, is there a natural coordinate related to the spectral gap?
\end{question}

\subsection{Computational Aspects}

\begin{question}[Algorithmic Coordinate Selection]\label{q:algorithm}
Given samples $X_1, \ldots, X_n$ from an unknown distribution, can we efficiently estimate the optimal concentration coordinate $\psi^*$?
\end{question}

\begin{conjecture}[Adaptive Concentration]\label{conj:adaptive}
There exists a data-driven procedure $\hat{\psi}_n$ based on $n$ samples such that:
\begin{equation}
\Fc[\mu, \hat{\psi}_n] \leq \inf_\psi \Fc[\mu, \psi] + O(n^{-1/2}).
\end{equation}
\end{conjecture}

\subsection{Connections to Curvature}

\begin{question}[Ricci Curvature and Optimal Coordinates]\label{q:ricci}
Is the optimal concentration coordinate related to the coordinate system that maximizes the Bakry-Émery Ricci curvature? Can we characterize $\psi^*$ in terms of:
\begin{equation}
\psi^* = \argmax_\psi \inf_{x \in E} \Ric^\psi_\mu(x)?
\end{equation}
\end{question}

\section*{Acknowledgments}

The author thanks the anonymous referees for their valuable comments.

\bibliographystyle{plainnat}

\end{document}